\documentclass[11pt]{amsart}

\usepackage[T1]{fontenc}
\usepackage{lmodern}
\usepackage{microtype}
\usepackage{amsmath,amssymb,amsthm,mathtools}
\usepackage{booktabs}
\usepackage{enumitem}
\usepackage[hidelinks]{hyperref}
\usepackage[margin=1.1in]{geometry}

\newtheorem{theorem}{Theorem}[section]
\newtheorem{proposition}[theorem]{Proposition}
\newtheorem{lemma}[theorem]{Lemma}
\newtheorem{corollary}[theorem]{Corollary}
\theoremstyle{remark}
\newtheorem{remark}[theorem]{Remark}

\newcommand{\Pp}{\mathbb{P}}
\newcommand{\Ee}{\mathbb{E}}
\newcommand{\Vv}{\operatorname{Var}}
\newcommand{\Pois}{\operatorname{Pois}}
\newcommand{\Stirling}[2]{\genfrac\{\}{0pt}{}{#1}{#2}}

\title[The Last Isolated Vertex on Cycle Powers]{The Last Isolated Vertex in Random-Order Uncovering of Cycles and Their Powers:\\Exact Enumeration and Weibull Limits}
\author{Igor Kleiner}
\address{Department of Data Science, HIT -- Holon Institute of Technology, Holon, Israel; University of Haifa, Haifa, Israel}
\email{igorkl@hit.ac.il}
\date{}

\begin{document}

\begin{abstract}
Let the vertices of a graph be revealed one at a time in a uniformly random order, and let the last-isolation time be the last time at which the induced graph on the revealed vertices contains an isolated vertex.  For the cycle $C_n$, write $K_n$ for the number of vertices still unrevealed at this time.  We obtain an exact finite-$n$ tail formula for $K_n$ in terms of Stirling numbers of the second kind, together with a compact bivariate generating function.

The enumeration comes from reversing the process.  An isolated revealed vertex then corresponds to the cyclic pattern $101$, but the tail event requires this pattern to be absent from every earlier prefix, not merely from the final set.  This prefix condition forces the final one-blocks to be separated by zero-gaps of length at least two and forces the reveal order inside each one-block to be peakless.  Since a block of size $m$ has $2^{m-1}$ peakless orders, summing over block sizes produces the Stirling numbers.

We also prove
\[
 \frac{K_n}{\sqrt n}\xrightarrow{d}R,\qquad \Pp(R>x)=e^{-x^2},
\]
and, for every fixed $d\ge1$,
\[
 \frac{K_{n,d}}{n^{1-1/(2d)}}\xrightarrow{d}W_d,
 \qquad \Pp(W_d>x)=e^{-x^{2d}}.
\]
Uniform stretched-exponential tail bounds imply convergence of all fixed positive moments.  Thus the ordinary cycle has a Rayleigh limit, while its fixed powers give a Weibull family with shape parameter $2d$.
\end{abstract}

\maketitle

\section{Introduction}

Let a fixed graph be uncovered by revealing its vertices one at a time in a uniformly random order, retaining at each time the subgraph induced by the revealed vertices.  We study the last time at which this induced graph contains an isolated vertex.  Our exact result concerns the cycle $C_n$, and the asymptotic analysis extends to every fixed power $C_n^d$.

Random-order vertex exposure has appeared in several related settings.  Newman and Ziff, for example, use a random ordering of sites in their site-percolation algorithm \cite{NewmanZiff2001}.  In the language of randomly evolving binary sequences, the classical ``island'' problem concerns occupied runs.  Janson traces this problem back to H\"allstr\"om's 1952 paper \cite{Hallstrom1952} and reports that H\"allstr\"om studied the maximum number of islands, including a cyclic version \cite{Janson2009Runs}.  Janson also studies cyclic runs of length one and the maximum of their count over the evolution \cite{Janson2009Runs}.  We ask a different question: when is that singleton count nonzero for the last time?  For singleton runs, the maximum is a bulk statistic, attained around one third of the evolution, whereas the stopping time studied here lies near the end.

Janson's recent work on \emph{uncovering a graph} studies the same random-order vertex process in a general setting, with particular emphasis on the evolution of visible edges \cite{Janson2026Uncovering}.  For trees it also treats the number of components in the visible subgraph, and the method is extended to counts of other small subgraphs.  Random-order exposure of powers of directed paths appears in the best-choice problem of Benevides and Su{\l}kowska \cite{BenevidesSulkowska2017}.  Related optimal-stopping problems ask when to stop so as to maximize the number of components in a random-order induced subgraph \cite{Lason2021,BenevidesSulkowska2022}.

For the ordinary cycle, reversing the uncovering order converts the event into the first occurrence of the cyclic pattern $101$.  Requiring this pattern to be absent at every earlier prefix imposes two constraints on the final reverse set: zero-gaps have length at least two, and the relative reveal order inside each one-block has no internal peak.  Since a block of size $m$ has $2^{m-1}$ peakless orders, summing over block sizes introduces Stirling numbers of the second kind and yields the exact finite-$n$ distribution.  Peaks and prescribed peak sets of permutations are classical objects \cite{BilleyBurdzySagan2013}.  Here peaklessness is not imposed as a separate permutation restriction; it is forced by requiring every earlier prefix to avoid $101$.

For powers of cycles, the limiting law is governed by rare completed neighborhoods in the reversed process.  We handle these locally dependent rare events with factorial moments, using the overlap geometry of $C_n^d$, and then compare the completed-neighborhood time with the true last-isolation time.  This gives a Weibull limit with shape $2d$ and convergence of all fixed positive moments.  For the classical Poisson-approximation framework for locally dependent events, see Arratia, Goldstein and Gordon \cite{ArratiaGoldsteinGordon1989}.

A last-isolated-vertex stopping time also appears in random graph processes with random edge exposure \cite{GlebovLuriaSimkin2021}; there the randomness is in the edges rather than in the vertex order.

The main results are as follows.
\begin{enumerate}[label=\textup{(\roman*)}]
 \item For $C_n$, we obtain the exact tail law
 \[
 \Pp(K_n>k)=
 \frac{(n-k)!}{(n-1)!}
 \sum_{r=1}^{\min\{k,\lfloor(n-k)/2\rfloor\}}
 2^{k-r}(r-1)!\Stirling{k}{r}
 \binom{n-k-r-1}{r-1},
 \]
 for $1\le k\le n-1$, together with a compact bivariate generating function.
 \item We prove
 \[
   K_n/\sqrt n\xrightarrow{d}R,
   \qquad \Pp(R>x)=e^{-x^2},
 \]
 and convergence of every fixed positive moment.
 \item For every fixed $d\ge1$, we prove
 \[
   \frac{K_{n,d}}{n^{1-1/(2d)}}\xrightarrow{d}W_d,
   \qquad \Pp(W_d>x)=e^{-x^{2d}},
 \]
 again with convergence of every fixed positive moment.
\end{enumerate}
To our knowledge, this last-isolation time has not been studied before, and the exact finite-$n$ distribution obtained here appears to be new.  The Rayleigh and Weibull limits concern this stopping statistic; their proofs use standard rare-event and Poisson-approximation ideas.

\section{Model and reverse-process representation}

Let $C_n$ be the cycle on vertex set $\mathbb Z/n\mathbb Z$, with $n\ge 3$.  Let
\[
 \pi=(\pi_1,\ldots,\pi_n)
\]
be a uniformly random permutation of its vertices and set
\[
 A_t=\{\pi_1,\ldots,\pi_t\},\qquad 0\le t\le n.
\]
A vertex $v\in A_t$ is called \emph{isolated at time $t$} if neither neighbor of $v$ lies in $A_t$.  Define
\[
 \tau_n=\max\{t: C_n[A_t]\text{ contains an isolated vertex}\}
\]
and
\[
 K_n=n-\tau_n.
\]
Thus $K_n$ is the number of vertices still unrevealed when an isolated revealed vertex is seen for the last time.

Reverse the random order.  Write
\[
 \rho_j=\pi_{n-j+1},\qquad
 U_k=\{\rho_1,\ldots,\rho_k\}.
\]
At forward time $n-k$, the unrevealed set is precisely $U_k$.

For a subset $U\subseteq V(C_n)$, encode membership by a cyclic binary word, with $1$ denoting membership in $U$.  We say that $U$ contains the cyclic pattern $101$ if there exists $v$ such that
\[
 v-1,v+1\in U,\qquad v\notin U.
\]

\begin{lemma}[First $101$ representation]\label{lem:first101}
For every $n\ge3$,
\[
 K_n=\min\{k\ge 0:U_k\text{ contains the cyclic pattern }101\}.
\]
Consequently,
\[
 \{K_n>k\}=\{U_j\text{ contains no }101\text{ for every }j\le k\}.
\]
\end{lemma}

\begin{proof}
At forward time $n-k$, a revealed vertex $v$ is isolated exactly when both of its neighbors remain unrevealed.  Since the unrevealed set is $U_k$, this is equivalent to
\[
 v\notin U_k,\qquad v-1,v+1\in U_k,
\]
which is exactly a cyclic $101$.  Maximizing the forward time is therefore equivalent to minimizing the reverse time.
\end{proof}

\section{Exact enumeration}

Fix $k$ and condition on the final reverse set $U_k=S$.  Suppose the $1$'s of $S$ form $r$ cyclic blocks with lengths
\[
 m_1,\ldots,m_r,
 \qquad m_i\ge1,
 \qquad \sum_{i=1}^r m_i=k,
\]
and let the intervening $0$-block lengths be
\[
 z_1,\ldots,z_r,
 \qquad \sum_{i=1}^r z_i=n-k.
\]

\begin{lemma}[Safe prefixes]\label{lem:safe}
The event $K_n>k$ holds for a given ordered $k$-prefix if and only if both of the following conditions hold:
\begin{enumerate}[label=\textup{(\roman*)}]
 \item every zero-block has length at least two, $z_i\ge2$;
 \item within each one-block, the relative order of reverse reveal times has no internal peak.
\end{enumerate}
\end{lemma}

\begin{proof}
If some $z_i=1$, the final set $U_k$ itself contains a $101$, so $K_n\le k$.

Now consider three consecutive vertices inside one final one-block.  Let their reverse reveal times be $a_{i-1},a_i,a_{i+1}$.  If
\[
 a_i>\max(a_{i-1},a_{i+1}),
\]
then immediately after both neighbors have appeared and before the center appears, the reverse set contains $101$.  Hence every block order must have no internal peak.

Conversely, suppose a $101$ occurs at some prefix $U_j$, centered at $v$.  If $v\notin S$, then in the final set the two neighboring one-blocks are separated by a single zero, contradicting (i).  If $v\in S$, then $v$ belongs to a final one-block but appears after both of its neighbors, giving an internal peak and contradicting (ii).  Thus (i) and (ii) are also sufficient.
\end{proof}

\begin{lemma}[Peakless permutations]\label{lem:peakless}
The number of permutations of $[m]$ with no internal peak is
\[
 2^{m-1}.
\]
\end{lemma}

\begin{proof}
In a permutation with no internal peak, the maximum entry $m$ must lie at one of the two ends.  Removing it leaves a peakless permutation of $[m-1]$.  Conversely, adjoining $m$ to either end of a peakless permutation creates no internal peak.  Thus, if $p_m$ denotes the number of peakless permutations,
\[
 p_m=2p_{m-1},\qquad p_1=1,
\]
which gives $p_m=2^{m-1}$.
\end{proof}

For fixed block sizes $m_1,\ldots,m_r$, Lemma \ref{lem:peakless} shows that the number of globally ordered $k$-prefixes consistent with those blocks and safe inside every block is
\[
 \frac{k!}{m_1!\cdots m_r!}\prod_{i=1}^r2^{m_i-1}.
\]
Introduce
\[
 A(x)=\sum_{m\ge1}\frac{2^{m-1}}{m!}x^m
      =\frac{e^{2x}-1}{2}.
\]
The standard exponential generating function for Stirling numbers of the second kind gives
\begin{equation}\label{eq:stirlingcoef}
 [x^k]A(x)^r
 =\frac{2^{k-r}r!}{k!}\Stirling{k}{r}.
\end{equation}
The number of compositions of $n-k$ into $r$ zero-blocks of size at least two is
\begin{equation}\label{eq:zeroblocks}
 \binom{n-k-r-1}{r-1}.
\end{equation}
Finally, sum over all ordered lists of one-block and zero-block sizes with $r$ one-blocks, and root a cyclic configuration at a distinguished first vertex of a one-block.  There are $n$ choices for the root.  Forgetting which one-block was distinguished identifies exactly $r$ rooted descriptions of every unrooted cyclic configuration.  Thus, after summing over the ordered block data, the cyclic factor is $n/r$.

\begin{theorem}[Exact tail distribution]\label{thm:exact}
For $n\ge3$ and $1\le k\le n-1$,
\begin{equation}\label{eq:exacttail}
\boxed{
 \Pp(K_n>k)=
 \frac{(n-k)!}{(n-1)!}
 \sum_{r=1}^{\min\{k,\lfloor(n-k)/2\rfloor\}}
 2^{k-r}(r-1)!\Stirling{k}{r}
 \binom{n-k-r-1}{r-1}.}
\end{equation}
Hence
\[
 \Pp(K_n=k)=\Pp(K_n>k-1)-\Pp(K_n>k).
\]
\end{theorem}

\begin{proof}
For a fixed number $r$ of one-blocks, multiply the cyclic factor $n/r$, the zero-block count \eqref{eq:zeroblocks}, and $k!$ times the coefficient in \eqref{eq:stirlingcoef}.  The number of safe ordered $k$-prefixes with $r$ one-blocks is therefore
\[
 n\,2^{k-r}(r-1)!\Stirling{k}{r}
 \binom{n-k-r-1}{r-1}.
\]
There are $(n)_k=n!/(n-k)!$ ordered $k$-prefixes in total.  Summing over all feasible $r$ and dividing by $(n)_k$ gives \eqref{eq:exacttail}.
\end{proof}

\begin{corollary}[A bivariate generating function]\label{cor:gf}
For $0\le k\le n$, let
\[
 Q_{n,k}=\binom nk\Pp(K_n>k).
\]
Then $Q_{n,0}=1$, and for $1\le k\le n$,
\[
 \boxed{
 Q_{n,k}
 =n[x^ky^{n-k}]
 \log\!\left(
 \frac{1}{1-\frac{e^{2x}-1}{2}\frac{y^2}{1-y}}
 \right).}
\]
\end{corollary}

\begin{proof}
A one-block contributes $A(x)=(e^{2x}-1)/2$, and a zero-block of length at least two contributes $y^2/(1-y)$.  A cyclic sequence of $r$ alternating one-blocks and zero-blocks carries the factor $n/r$.  The resulting coefficient counts safe ordered $k$-prefixes divided by $k!$; since $(n)_k/k!=\binom nk$, this is exactly $Q_{n,k}$.  Summing over $r\ge1$ yields the logarithm.
\end{proof}

\begin{remark}[An OEIS connection]\label{rem:oeis}
Let $A_{n,k}:=(n)_k\Pp(K_n>k)$, the number of safe ordered reverse prefixes of length $k$.  Then $A_{n,n-2}=n2^{n-3}$ for $n\ge3$ and $A_{n,n-3}=n2^{n-4}$ for $n\ge4$; after the index shift $m=n-2$, these are OEIS A001792 and A087447, respectively.  This observation is not used below.
\end{remark}

\section{The Rayleigh limit}

The exact formula is useful for finite $n$, but the limiting distribution has a simpler probabilistic proof.

Construct an auxiliary graph $H_n$ on the same vertex set as $C_n$ by joining vertices at cyclic distance two.  For $n\ge5$, $H_n$ has $n$ edges and is
\[
 H_n\cong
 \begin{cases}
 C_n,&n\text{ odd},\\
 C_{n/2}\sqcup C_{n/2},&n\text{ even}.
 \end{cases}
\]
Define
\[
 J_n=\min\{k:U_k\text{ contains an edge of }H_n\}.
\]
Clearly $J_n\le K_n$.

\begin{lemma}[Poisson collision limit]\label{lem:poisson}
For every fixed $x\ge0$, if $k=\lfloor x\sqrt n\rfloor$ and
\[
 X_{n,k}=\#\{e\in E(H_n):e\subseteq U_k\},
\]
then
\[
 X_{n,k}\xrightarrow{d}\Pois(x^2).
\]
Consequently,
\[
 \Pp(J_n>x\sqrt n)\longrightarrow e^{-x^2}.
\]
\end{lemma}

\begin{proof}
For a fixed edge $e$,
\[
 \Pp(e\subseteq U_k)=\frac{(k)_2}{(n)_2},
\]
so
\[
 \Ee X_{n,k}=n\frac{(k)_2}{(n)_2}\longrightarrow x^2.
\]
We verify convergence of factorial moments.  Fix $r$.  In $\Ee[(X_{n,k})_r]$, ordered $r$-tuples of pairwise vertex-disjoint edges number $n^r+O(n^{r-1})$, and each such tuple is present with probability
\[
 \frac{(k)_{2r}}{(n)_{2r}}.
\]
Their contribution therefore tends to $x^{2r}$.

It remains to bound tuples containing an overlap.  Since $r$ is fixed, for all sufficiently large $n$ (in particular, $n>2r$) no selected set of at most $r$ edges can contain an entire cycle component of $H_n$.  Thus the graph formed by the distinct selected edges is a disjoint union of nontrivial paths.  If it has $c$ connected components and $v$ distinct vertices, then
\[
 v=r+c.
\]
For each fixed overlap type there are $O(n^c)$ placements in $H_n$, while the probability that all of its $v$ vertices lie in $U_k$ is
\[
 \frac{(k)_v}{(n)_v}=O\!\left((k/n)^v\right).
\]
Hence its total contribution is
\[
 O(n^c)\,O\!\left((k/n)^v\right)
 =O\!\left(n^{c-v/2}\right)
 =O\!\left(n^{(c-r)/2}\right).
\]
For an overlapping $r$-tuple, $c\le r-1$, so this is $O(n^{-1/2})$.  There are only finitely many overlap types for fixed $r$, and therefore all overlapping tuples together contribute $o(1)$.  Hence
\[
 \Ee[(X_{n,k})_r]\longrightarrow x^{2r},
\]
the factorial moments of $\Pois(x^2)$.  The standard factorial-moment criterion therefore gives
$X_{n,k}\xrightarrow{d}\Pois(x^2)$.  The final assertion follows from
\[
 \{J_n>k\}=\{X_{n,k}=0\}.
\]
\end{proof}

\begin{lemma}[The proxy and the true time agree at scale $\sqrt n$]\label{lem:proxy}
For fixed $x\ge0$ and $k=\lfloor x\sqrt n\rfloor$,
\[
 0\le \Pp(K_n>k)-\Pp(J_n>k)
 \le n\frac{(k)_3}{(n)_3}
 =O_x(n^{-1/2}).
\]
\end{lemma}

\begin{proof}
If $J_n\le k<K_n$, then at the first time a distance-two pair $v-1,v+1$ appears, the center $v$ must already be present; otherwise that pair would create $101$ and force $K_n=J_n$.  Thus $U_k$ contains all three vertices $v-1,v,v+1$ for some $v$.  A union bound over the $n$ cyclic triples gives the stated bound.
\end{proof}

\begin{theorem}[Rayleigh limit]\label{thm:rayleigh}
As $n\to\infty$,
\[
 \boxed{
 \frac{K_n}{\sqrt n}\xrightarrow{d}R,
 \qquad \Pp(R>x)=e^{-x^2},\quad x\ge0.}
\]
Equivalently, $R$ is Rayleigh with scale parameter $1/\sqrt2$.
\end{theorem}

\begin{proof}
Combine Lemmas \ref{lem:poisson} and \ref{lem:proxy}.
\end{proof}

\section{Tail bounds and convergence of moments}

Convergence in distribution alone does not imply convergence of expectations.  We therefore establish a uniform tail bound.

Represent the reverse order by i.i.d. priorities $Y_v\sim\mathrm{Unif}(0,1)$, revealed in increasing order.  Choose $m=\lfloor n/3\rfloor$ pairwise vertex-disjoint triples of consecutive vertices $(a_i,b_i,c_i)$; for instance, use $(3j,3j+1,3j+2)$ for $0\le j<m$.  For $1\le k\le n$, put
\[
 p=\frac{k}{2n}
\]
and define
\[
 E_i=\{Y_{a_i}\le p,\;Y_{c_i}\le p,\;Y_{b_i}>\max(Y_{a_i},Y_{c_i})\}.
\]
A direct integration gives
\begin{equation}\label{eq:q}
 \Pp(E_i)
 =\int_0^p\int_0^p(1-\max(u,v))\,du\,dv
 =p^2-\frac23p^3.
\end{equation}
For $p\le1/2$, this is at least $(2/3)p^2$.  The events $E_i$ are independent.

Let
\[
 N_p=\#\{v:Y_v\le p\}\sim\operatorname{Bin}(n,p),
\]
so that $\Ee N_p=k/2$.  If $N_p\le k$ and at least one $E_i$ occurs, then a $101$ appears no later than reverse step $k$.  Therefore
\[
 \{K_n>k\}\subseteq\{N_p>k\}\cup\bigcap_{i=1}^mE_i^c.
\]

\begin{proposition}[Uniform tail bound]\label{prop:tail}
There is a universal constant $c_0>0$ such that, for every $n\ge6$ and every $1\le k\le n$,
\[
 \boxed{
 \Pp(K_n>k)
 \le e^{-c_0k}+\exp\!\left(-\frac{k^2}{24n}\right).}
\]
One may take $c_0=(\log4-1)/2$.
\end{proposition}

\begin{proof}
A standard multiplicative Chernoff bound, with $\Ee N_p=k/2$, gives
\[
 \Pp(N_p>k)\le(e/4)^{k/2}=e^{-c_0k}.
\]
For $n\ge6$, $m\ge n/4$.  Using \eqref{eq:q} and $p=k/(2n)$,
\[
 m\Pp(E_i)\ge \frac n4\cdot\frac23\cdot\frac{k^2}{4n^2}
 =\frac{k^2}{24n}.
\]
Independence gives
\[
 \Pp\!\left(\bigcap_iE_i^c\right)
 \le \exp(-m\Pp(E_i)),
\]
which proves the claim.
\end{proof}

\begin{theorem}[Convergence of all fixed moments]\label{thm:moments}
For every fixed $r>0$,
\[
 \boxed{
 \Ee\left[\left(\frac{K_n}{\sqrt n}\right)^r\right]
 \longrightarrow
 \Gamma\!\left(1+\frac r2\right).}
\]
In particular,
\[
 \Ee K_n\sim\frac{\sqrt\pi}{2}\sqrt n
\]
and
\[
 \Vv(K_n)\sim\left(1-\frac\pi4\right)n.
\]
Equivalently,
\[
 \Ee\tau_n
 =n-\frac{\sqrt\pi}{2}\sqrt n+o(\sqrt n).
\]
\end{theorem}

\begin{proof}
Let $Z_n=K_n/\sqrt n$ and fix $n\ge6$.  Since $K_n\le n-1$, we have $\Pp(Z_n>x)=0$ whenever $x\ge\sqrt n$.  Thus it suffices to consider $0\le x<\sqrt n$, for which $k=\lfloor x\sqrt n\rfloor$ satisfies $0\le k\le n-1$ and
\[
 \Pp(Z_n>x)=\Pp(K_n>k).
\]
If $x\sqrt n\ge2$, then $k\ge x\sqrt n/2$, and Proposition \ref{prop:tail} gives
\[
 \Pp(Z_n>x)
 \le
 \exp\!\left(-\frac{c_0x\sqrt n}{2}\right)
 +\exp\!\left(-\frac{x^2}{96}\right)
 \le
 \exp(-c_1x)+\exp\!\left(-\frac{x^2}{96}\right),
\]
where $c_1=c_0\sqrt6/2>0$.  When $x\sqrt n<2$ we simply use $\Pp(Z_n>x)\le1$.  Enlarging constants if necessary therefore yields universal $C,c,c'>0$ such that
\[
 \Pp(Z_n>x)\le C e^{-cx}+C e^{-c'x^2},
 \qquad x\ge0,
\]
uniformly for $n\ge6$.  The finitely many smaller values of $n$ are immaterial.

For any fixed $r>0$, the function
\[
 x\longmapsto r x^{r-1}\bigl(Ce^{-cx}+Ce^{-c'x^2}\bigr)
\]
is integrable on $(0,\infty)$.  Theorem \ref{thm:rayleigh} gives pointwise convergence of the tails at every $x\ge0$, so the tail-integral identity and dominated convergence yield
\[
 \Ee Z_n^r
 =r\int_0^\infty x^{r-1}\Pp(Z_n>x)\,dx
 \longrightarrow
 r\int_0^\infty x^{r-1}e^{-x^2}\,dx
 =\Gamma\!\left(1+\frac r2\right).
\]
The displayed asymptotics follow by taking $r=1,2$.
\end{proof}

\section{Powers of cycles and a Weibull hierarchy}\label{sec:powers}

The Rayleigh law extends to a Weibull family for powers of cycles.  Fix an integer $d\ge1$.  For $n>2d+1$, let $C_n^d$ denote the $d$th power of the cycle: two vertices are adjacent when their cyclic distance is at most $d$.  Reveal its vertices in a uniformly random order and let $\tau_{n,d}$ be the last time at which the induced graph on the revealed vertices contains an isolated vertex.  Put
\[
 K_{n,d}=n-\tau_{n,d}.
\]
Thus $K_{n,1}=K_n$.

In the reversed process, with $U_k$ denoting the first $k$ reverse-revealed vertices, write
\[
 N_d(v)=\{v-d,\ldots,v-1,v+1,\ldots,v+d\}
\]
for the $2d$ neighbors of $v$ in $C_n^d$.  Then
\begin{equation}\label{eq:powerK}
 K_{n,d}=\min\{k:\exists v\notin U_k\text{ with }N_d(v)\subseteq U_k\}.
\end{equation}
Introduce the simpler proxy
\[
 J_{n,d}=\min\{k:\exists v\text{ with }N_d(v)\subseteq U_k\},
\]
so that $J_{n,d}\le K_{n,d}$.

Set
\[
 a_{n,d}=n^{1-1/(2d)}.
\]

\begin{lemma}[Poisson neighborhood limit]\label{lem:powerpoisson}
For every fixed $d\ge1$ and $x\ge0$, let $k=\lfloor x a_{n,d}\rfloor$ and define
\[
 X_{n,k}^{(d)}=\#\{v:N_d(v)\subseteq U_k\}.
\]
Then
\[
 X_{n,k}^{(d)}\xrightarrow{d}\Pois(x^{2d}).
\]
Consequently,
\[
 \Pp(J_{n,d}>x a_{n,d})\longrightarrow e^{-x^{2d}}.
\]
\end{lemma}

\begin{proof}
We use factorial moments.  For one fixed neighborhood,
\[
 \Pp(N_d(v)\subseteq U_k)=\frac{(k)_{2d}}{(n)_{2d}},
\]
and hence
\[
 \Ee X_{n,k}^{(d)}
 =n\frac{(k)_{2d}}{(n)_{2d}}
 \longrightarrow x^{2d}.
\]
Fix a positive integer $q$.  In $\Ee[(X_{n,k}^{(d)})_q]$, ordered $q$-tuples of pairwise disjoint neighborhoods number
\[
 n^q+O_{d,q}(n^{q-1}),
\]
because each $N_d(v)$ intersects only $O_d(1)$ other neighborhoods (an intersection forces the two centers to be at cyclic distance at most $2d$).  Their contribution to the $q$th factorial moment is therefore
\[
 \bigl(n^q+O_{d,q}(n^{q-1})\bigr)
 \frac{(k)_{2dq}}{(n)_{2dq}}
 \longrightarrow x^{2dq}.
\]

It remains to bound tuples containing at least one overlap.  Build the overlap graph on the selected neighborhoods, joining two selected neighborhoods when they intersect.  Suppose this overlap graph has $c$ connected components, of which $h\ge1$ are non-singletons.  For all sufficiently large $n$, distinct neighborhoods $N_d(v)$ are distinct $2d$-sets.  Hence a singleton component contributes exactly $2d$ vertices to the union, while every non-singleton component contributes at least $2d+1$.  If $s$ is the total number of distinct vertices in the union, then
\[
 s\ge 2dc+h.
\]
For fixed $d$ and $q$, a connected overlap component has only $O_{d,q}(1)$ possible relative shapes once one center is fixed, since intersecting neighborhoods have centers at bounded cyclic distance.  Thus a tuple with $c$ overlap components has $O_{d,q}(n^c)$ placements.  Its total contribution is at most
\[
 O_{d,q}(n^c)\,O\!\left((k/n)^s\right)
 =O_{d,q}\!\left(n^{c-s/(2d)}\right)
 \le O_{d,q}\!\left(n^{-h/(2d)}\right)
 =o(1).
\]
There are only finitely many overlap types for fixed $d$ and $q$.  Therefore
\[
 \Ee\!\left[(X_{n,k}^{(d)})_q\right]\longrightarrow x^{2dq},
\]
the factorial moments of $\Pois(x^{2d})$.  The standard factorial-moment criterion therefore gives
$X_{n,k}^{(d)}\xrightarrow{d}\Pois(x^{2d})$.  The assertion for $J_{n,d}$ follows from
\[
 \{J_{n,d}>k\}=\{X_{n,k}^{(d)}=0\}.
\]
\end{proof}

\begin{remark}[Poisson-approximation viewpoint]\label{rem:chenstein}
The family $\{N_d(v):v\in V(C_n^d)\}$ is a $2d$-uniform family with bounded local overlap: each neighborhood intersects only $O_d(1)$ others.  Thus Lemma \ref{lem:powerpoisson} is a particularly simple instance of the classical Poisson-approximation paradigm for rare, locally dependent events \cite{ArratiaGoldsteinGordon1989}.  We use factorial moments because they expose the cycle-power overlap geometry and the limiting parameter $x^{2d}$ directly.
\end{remark}

\begin{lemma}[Proxy error]\label{lem:powerproxy}
For fixed $d\ge1$, $x\ge0$, and $k=\lfloor x a_{n,d}\rfloor$,
\[
 0\le \Pp(K_{n,d}>k)-\Pp(J_{n,d}>k)
 \le n\frac{(k)_{2d+1}}{(n)_{2d+1}}
 =O_{d,x}(n^{-1/(2d)}).
\]
\end{lemma}

\begin{proof}
If $J_{n,d}\le k<K_{n,d}$, then at the first time some neighborhood $N_d(v)$ is fully present in the reverse set, its center $v$ must already be present as well; otherwise $v$ would be an isolated revealed vertex in the forward process and \eqref{eq:powerK} would give $K_{n,d}=J_{n,d}$.  Thus $U_k$ contains the entire closed neighborhood $N_d(v)\cup\{v\}$, a set of size $2d+1$.  A union bound over the $n$ possible centers gives the displayed estimate.  Since $k/n=O(n^{-1/(2d)})$, the final order follows.
\end{proof}

\begin{theorem}[Weibull hierarchy]\label{thm:weibull}
For every fixed integer $d\ge1$,
\[
 \boxed{
 \frac{K_{n,d}}{n^{1-1/(2d)}}\xrightarrow{d}W_d,
 \qquad
 \Pp(W_d>x)=e^{-x^{2d}},\quad x\ge0.}
\]
Thus $W_d$ is Weibull with shape parameter $2d$ and scale $1$.  The case $d=1$ is the Rayleigh law of Theorem \ref{thm:rayleigh}.
\end{theorem}

\begin{proof}
Combine Lemmas \ref{lem:powerpoisson} and \ref{lem:powerproxy}.
\end{proof}

The convergence extends to all fixed positive moments.  We record a tail estimate that will also make the required uniform integrability transparent.

\begin{proposition}[Uniform stretched-exponential tail]\label{prop:powertail}
Fix $d\ge1$.  There are constants $c_0>0$ and $c_d>0$ such that, for all sufficiently large $n$ and all $1\le k\le n$,
\[
 \boxed{
 \Pp(K_{n,d}>k)
 \le e^{-c_0k}
 +\exp\!\left(-c_d\frac{k^{2d}}{n^{2d-1}}\right).}
\]
One may take
\[
 c_0=\frac{\log4-1}{2},
 \qquad
 c_d=\frac{1}{2^{2d+2}(2d+1)}.
\]
\end{proposition}

\begin{proof}
Assign i.i.d. priorities $Y_v\sim\mathrm{Unif}(0,1)$ and reveal vertices in increasing priority order.  Choose
\[
 m=\left\lfloor\frac{n}{2d+1}\right\rfloor
\]
centers whose closed radius-$d$ neighborhoods are pairwise disjoint.  One explicit choice, using representatives $0,\ldots,n-1$, is
\[
 v_j=d+j(2d+1),\qquad 0\le j<m.
\]
For one such center $v$, put $p=k/(2n)$ and let $E_v$ be the event that all $2d$ neighbors of $v$ have priority at most $p$, while $Y_v$ is larger than all of those neighbor priorities.  Direct integration gives
\[
 \Pp(E_v)
 =p^{2d}-\frac{2d}{2d+1}p^{2d+1}
 \ge \frac12 p^{2d},
\]
since $p\le1/2$.  The chosen events are independent.

Let
\[
 N_p=\#\{v:Y_v\le p\}\sim\operatorname{Bin}(n,p),
\]
so $\Ee N_p=k/2$.  If $N_p\le k$ and at least one $E_v$ occurs, then at some reverse time at most $k$ all neighbors of that center are present before the center, so an isolated vertex has already been created.  Therefore
\[
 \{K_{n,d}>k\}
 \subseteq
 \{N_p>k\}\cup\bigcap_v E_v^c.
\]
The same Chernoff bound as before gives
\[
 \Pp(N_p>k)\le(e/4)^{k/2}=e^{-c_0k}.
\]
For $n\ge2(2d+1)$,
\[
 m\ge\frac{n}{2(2d+1)},
\]
and hence
\[
 m\Pp(E_v)
 \ge
 \frac{n}{2(2d+1)}\cdot\frac12
 \left(\frac{k}{2n}\right)^{2d}
 =c_d\frac{k^{2d}}{n^{2d-1}}.
\]
Independence completes the proof.
\end{proof}

\begin{theorem}[Moments for powers of cycles]\label{thm:powermoments}
For every fixed $d\ge1$ and every fixed $r>0$,
\[
 \boxed{
 \Ee\!\left[
 \left(\frac{K_{n,d}}{n^{1-1/(2d)}}\right)^r
 \right]
 \longrightarrow
 \Gamma\!\left(1+\frac{r}{2d}\right).}
\]
Consequently,
\[
 \Ee K_{n,d}
 \sim
 \Gamma\!\left(1+\frac{1}{2d}\right)n^{1-1/(2d)}
\]
and
\[
 \Vv(K_{n,d})
 \sim
 \left[
 \Gamma\!\left(1+\frac1d\right)
 -\Gamma\!\left(1+\frac{1}{2d}\right)^2
 \right]n^{2-1/d}.
\]
\end{theorem}

\begin{proof}
Let $Z_{n,d}=K_{n,d}/a_{n,d}$, where $a_{n,d}=n^{1-1/(2d)}$.  Since $K_{n,d}\le n-1$, the tail vanishes when
\[
 x\ge \frac{n}{a_{n,d}}=n^{1/(2d)}.
\]
It therefore suffices to consider $0\le x<n^{1/(2d)}$.  Put $k=\lfloor xa_{n,d}\rfloor$, so $0\le k\le n-1$.  If $xa_{n,d}\ge2$, then $k\ge xa_{n,d}/2$.  Proposition \ref{prop:powertail} and the identity $a_{n,d}^{2d}=n^{2d-1}$ give
\[
 \Pp(Z_{n,d}>x)
 \le
 \exp\!\left(-\frac{c_0xa_{n,d}}2\right)
 +\exp\!\left(-\frac{c_d}{2^{2d}}x^{2d}\right).
\]
For all sufficiently large $n$, $a_{n,d}\ge1$, so the first term is at most $e^{-c_0x/2}$.  When $xa_{n,d}<2$ we use the trivial bound $1$.  Hence there are constants $C_d,c,c'_d>0$, depending only on $d$, such that
\[
 \Pp(Z_{n,d}>x)
 \le C_de^{-cx}+C_de^{-c'_d x^{2d}},
 \qquad x\ge0,
\]
uniformly in all sufficiently large $n$.

This dominating tail is integrable against $r x^{r-1}$ for every fixed $r>0$.  Theorem \ref{thm:weibull} gives pointwise convergence of the tails, so dominated convergence and the tail-integral identity yield
\[
 \Ee Z_{n,d}^r
 \longrightarrow
 r\int_0^\infty x^{r-1}e^{-x^{2d}}\,dx
 =\Gamma\!\left(1+\frac{r}{2d}\right).
\]
The mean and variance formulas follow from $r=1,2$.
\end{proof}

\begin{remark}
The critical exponent has a simple rare-event interpretation.  When $k$ vertices remain unrevealed, a fixed revealed vertex of $C_n^d$ is isolated only if all of its $2d$ neighbors lie among those $k$ vertices, an event of order $(k/n)^{2d}$.  Balancing $n(k/n)^{2d}$ at constant order gives $k\asymp n^{1-1/(2d)}$.  The theorem identifies the full limiting law, not only this heuristic scale.
\end{remark}

\section{Computational verification}
The exact formula in Theorem \ref{thm:exact} was checked independently by exhaustive enumeration of every permutation for $3\le n\le11$ and by an exact subset dynamic program for every $n\le20$.  At $n=11$ the exhaustive check covers all $11!=39{,}916{,}800$ reveal orders.  Both procedures agree with the closed formula for every value of $k$ tested.  Independent small-case reverse/forward checks were also carried out for powers of cycles.  These computations are not used in the proofs; the verification code and run outputs are supplied as Supplementary Information.

\section{Discussion}

The exact law separates the problem into three combinatorial ingredients: cyclic zero-gaps of length at least two, peakless orders inside occupied blocks, and an exponential generating function whose powers produce Stirling numbers.  The asymptotic law, by contrast, is controlled only by the first distance-two collision in the reverse process; triple collisions are negligible at the $\sqrt n$ scale.  This is why the finite formula is more structured than the limiting law suggests.

For powers of cycles, the Rayleigh law at $d=1$ becomes a Weibull family arising from rare completed neighborhoods.  The model-specific step is the reduction of the last-isolation time to the neighborhood problem; for $d=1$, the exact enumeration also retains the prefix history that disappears from the limiting law.

\section*{Data and code availability}
No external data are used.  The computational verification code and run outputs are supplied as Supplementary Information; they include independent brute-force and subset-dynamic-programming checks of the exact finite-$n$ formula and exact small-case reverse/forward checks for powers of cycles.

\section*{Statements and declarations}
\noindent\textbf{Funding.} The author received no funding for this work.

\medskip
\noindent\textbf{Competing interests.} The author declares no competing interests.

\end{document}